\documentclass[journal,onecolumn,12pt]{article}

\usepackage{amssymb}
\usepackage{caption2}
\usepackage{amsmath}
\usepackage{multirow}
\usepackage{amsthm}
\usepackage{graphicx}
\usepackage{CJK}
\usepackage{enumerate}

\usepackage{subfigure}
\usepackage{makecell}

\usepackage{tikz}
\usetikzlibrary{shapes.geometric}
\tikzset{gon/.style={name=tmp,regular polygon,regular polygon sides=#1,minimum
size=10pt,inner sep=0pt},
polygon side/.style args={#1--#2}{
insert path={(tmp.corner #1)-- (tmp.corner #2)}}}
\newcommand{\FlagGraph}[3][]{\ifnum#2=2%
\tikz[baseline=(tmp1)]{\node[circle,inner sep=0.7pt,fill] (tmp1) at (0,0){};
\node[#1,circle,inner sep=0.7pt,fill] (tmp2) at (0,10pt){};
\ifx#3\empty%
\else
\draw[#1] (tmp1) -- (tmp2);
\fi}
\else%
\tikz[baseline=(tmp.south)]{\node[#1,gon=#2]{};
\foreach \X in {1,...,#2}{\fill (tmp.corner \X) circle (1pt);}
\draw[#1,polygon side/.list={#3}];}
\fi}

\newtheorem{theorem}{Theorem}[section]

\newtheorem{proposition}[theorem]{Proposition}
\newtheorem{corollary}[theorem]{Corollary}
\newtheorem{conjecture}[theorem]{Conjecture}
\newtheorem{definition}[theorem]{Definition}

\newtheorem{remark}[theorem]{Remark}

\begin{document}
\title{Some Results on k-Tur\'{a}n-good Graphs}

\author{Bingchen Qian$^{\text{a}},$ Chengfei Xie$^{\text{b}}$ and Gennian Ge$^{\text{b,}}$\thanks{
  Corresponding author. Email address: gnge@zju.edu.cn. Research supported by the National Natural Science Foundation of
China under Grant No. 11971325, National Key Research and Development Program of China under Grant Nos. 2020YFA0712100 and 2018YFA0704703,
and Beijing Scholars Program.}\\
\footnotesize $^{\text{a}}$ School of Mathematical Sciences, Zhejiang University, Hangzhou 310027, Zhejiang, China\\
\footnotesize $^{\text{b}}$ School of Mathematical Sciences, Capital Normal University, Beijing, 100048, China}

\maketitle

\begin{abstract}

For a graph $H$ and a $k$-chromatic graph $F,$ if the Tur\'an graph $T_{k-1}(n)$ has the maximum number of copies of $H$ among all $n$-vertex $F$-free graphs (for $n$ large enough), then $H$ is called $F$-Tur\'an-good, or $k$-Tur\'an-good for short if $F$ is $K_k.$ In this paper, we construct some new classes of $k$-Tur\'an-good graphs and prove that $P_4$ and $P_5$ are $k$-Tur\'an-good for $k\ge4.$  

  \medskip
  \noindent{\it Keywords:} Paths, Tur\'an graph, Tur\'an-good.

  \smallskip

  \noindent {{\it AMS subject classifications\/}:  05C35, 05C38.}
\smallskip
\end{abstract}

\section{Introduction}

For a fixed graph $F,$ we say that a graph $G$ is $F$-free if it contains no copy of $F$ as a subgraph. One of the most studied area of extremal combinatorics is Tur\'an theory, which seeks to determine $ex(n,F),$ the largest number of edges in an $F$-free graph on $n$ vertices. A cornerstone of extremal graph theory is Tur\'an's theorem \cite{turan1941external}, which determines the maximum number of edges in an $n$-vertex $K_k$-free graph. The extremal construction is a complete $(k-1)$-partite graph on $n$ vertices such that each class has cardinality either $\lceil\frac{n}{k-1}\rceil$ or $\lfloor\frac{n}{k-1}\rfloor.$ Such a graph is called a Tur\'an graph and is denoted by $T_{k-1}(n).$

 Alon and Shikhelman \cite{MR3548290} initialed the systematic study of the generalized Tur\'an problem $ex(n,H,F),$ the largest number of copies of $H$ in $F$-free graphs on $n$ vertices.
For several further recent results see, e.g., \cite{MR4103873,MR3996750,MR3989658,MR4123109,MR4012878}. One difficulty to determine $ex(n,H,F)$ exactly is that there are few $F$-free graphs that are good candidates for being extremal constructions for maximizing copies of a graph $H.$ An exception is the Tur\'an graph. In \cite{2020arXiv200603756G}, Gerbner and Palmer examined when the Tur\'an graph is the extremal construction for these generalized Tur\'an problems.

Fix graphs $H$ and $G,$ and denote the number of copies of $H$ in $G$ by $\mathcal{N}(H,G).$ Then $ex(n,H,F)=\max\{\mathcal{N}(H,G):G \text{ is an $n$-vertex $F$-free graph} \}.$ In what follows, $K_n$ is the complete graph on $n$ vertices. $C_n$ denotes the cycle on $n$ vertices, $P_n$ denotes the path on $n$ vertices (with $n-1$ edges) and $M_l$ denotes the matching with $l$ edges (thus $2l$ vertices). If $X\subseteq V(G),$ $G[X]$ denotes the subgraph of $G$ induced on $X.$ We say that an edge $e$ of a graph $G$ is color-critical if deleting $e$ from $G$ decreases the chromatic number of the graph.

\begin{definition}[\cite{2020arXiv200603756G}]
  Fix a $k$-chromatic graph $F$ and a graph $H$ that does not contain $F$ as a subgraph. We say that $H$ is $F$-Tur\'an-good if $ex(n,H,F)=\mathcal{N}(H,T_{k-1}(n))$ for every $n$ large enough. If $F=K_{k},$ we use the brief term $k$-Tur\'an-good.
\end{definition}

So Tur\'an's theorem states that $K_2$ is $k$-Tur\'an-good for every $k>2.$ An early result of Zykov \cite{MR0035428}
(see also Erd\H{o}s \cite{MR151956}) showed that $K_r$ is $k$-Tur\'an-good for any positive integer $r\le k-1.$
\begin{theorem}[\cite{MR0035428}]\label{Kr is k turan good}
  The Tur\'an graph $T_{k-1}(n)$ is the unique $n$-vertex $K_k$-free graph with the maximum number of copies of $K_r.$ Thus
  \begin{equation*}
    ex(n,K_r,K_k)=\mathcal{N}(K_r,T_{k-1}(n))\le{k-1\choose r}\lceil\frac{n}{k-1}\rceil^r.
  \end{equation*}
\end{theorem}

Theorem \ref{Kr is k turan good} was generalized by Ma and Qiu \cite{MR4014346} to graphs with a color-critical edge.
\begin{theorem}[\cite{MR4014346}]
  Let $F$ be a graph with a color-critical edge and chromatic number more than $r.$ Then $K_r$ is $F$-Tur\'an-good.
\end{theorem}

Gy\H{o}ri, Pach and Somonovits \cite{MR1105465} showed that a bipartite graph $H$ on $m\ge3$ vertices containing $\lfloor m/2\rfloor$ independent edges is $3$-Tur\'an-good. This implies that path $P_l$ with $l$ vertices, the even cycle $C_{2l}$ and the Tur\'an graph $T_2(m)$ are all $3$-Tur\'an-good. They also gave the following general theorem.

\begin{theorem}[\cite{MR1105465}]
Let $r\ge3,$ and let $H$ be a $(k-1)$-partite graph with $m > k-1$ vertices, containing $\lfloor m/(k-1)\rfloor$ vertex disjoint copies of $K_{k-1}.$ Suppose further that for any two vertices $u$ and $v$ in the same component of $H,$ there is a sequence $A_1,\ldots,A_s$ of $(k-1)$-cliques in $H$ such that $u\in A_1, v\in A_s,$ and for any $i < s,$ $A_i$ and $A_{i+1}$ share $k - 2$ vertices. Then $H$ is $k$-Tur\'an-good.
\end{theorem}

\begin{corollary}[\cite{MR1105465}]\label{paths are 3 Turan good}
  Paths, even cycles and $K_{2,3}$ are $3$-Tur\'an-good and $T_{k-1}(m)$ is $k$-Tur\'an-good.
\end{corollary}

Recently, Gerbner and Palmer \cite{2020arXiv200603756G} obtained a similar result, which gives a new class of $k$-Tur\'an-good graphs.
\begin{theorem}[\cite{2020arXiv200603756G}]\label{thm of H+K-1_is k turan good}
  Let $H$ be a $k$-Tur\'an-good graph. Let $H'$ be any graph constructed from $H$ in the following way. Choose a complete subgraph of $H$ with vertex set $X,$ add a vertex-disjoint copy of $K_{k-1}$ to $H$ and join the vertices in $X$ to the vertices of $K_{k-1}$ by edges arbitrarily. Then  $H'$ is $k$-Tur\'an-good.
\end{theorem}

And, Gerbner \cite{2020arXiv200616150G} also gave the following result.
\begin{theorem}[\cite{2020arXiv200616150G}]\label{Ml is k turan good}
  $M_l$ is $F$-Tur\'an-good for every $F$ with a color-critical edge.
\end{theorem}

For more results about $k$-Tur\'an-good graphs, we refer the readers to \cite{2020arXiv200616150G}, \cite{2020arXiv200603756G}, and their references.

Furthermore, Gerbner and Palmer raised the following conjectures.
\begin{conjecture}[\cite{2020arXiv200603756G}]\label{conj Pl is k turan good}
  For every pair of integers $l$ and $k,$ the path $P_l$ is $k$-Tur\'an-good.
\end{conjecture}
\begin{conjecture}[\cite{2020arXiv200603756G}]\label{conj pl and cl is cycle turan good}
  The path $P_k$ and the even cycle $C_{2k}$ are $C_{2l+1}$-Tur\'an-good.
\end{conjecture}
In the same paper, they made some progress, i.e., they proved that Conjecture \ref{conj Pl is k turan good} holds for $l=3$ and $k\ge3,$ and it holds asymptotically. The authors also showed that Conjecture \ref{conj pl and cl is cycle turan good} holds for $P_3$ when $l\ge1$ and for $P_4$ and $C_4$ when $l=2.$ They also obtained that if $P_{2k}$ is $k$-Tur\'an-good, then $C_{2k}$ is $k$-Tur\'an-good where $F$ is a $3$-chromatic graph. Recently, Gerbner \cite{2020arXiv201212646G} has solved Conjecture \ref{conj pl and cl is cycle turan good}.

In \cite{2020arXiv200616150G}, Gerbner proved the following case of Conjecture \ref{conj Pl is k turan good}.
\begin{proposition}[\cite{2020arXiv200616150G}]\label{P4 is 4 Turan good}
  $P_4$ is $4$-Tur\'an-good.
\end{proposition}

In this paper, we give some new $k$-Tur\'an-good graphs. And we also prove that $P_4$ and $P_5$ are $k$-Tur\'an-good for $k\ge4.$

The paper is organized as follows. In Section $2,$ we construct some new kinds of $k$-Tur\'an-good graphs. In Section $3,$ we prove that $P_4$ is $k$-Tur\'an-good. In Section $4,$ we prove that $P_5$ is $k$-Tur\'an-good. The last section contains some concluding remarks.

\section{Some new $k$-Tur\'an-good graphs}
For three graphs $G, G_1, G_2$, define $S_{G_1, G_2}(G)=\{(H_1, H_2):H_1$ and $H_2$ are disjoint subgraphs of $G, H_1$ is a copy of $G_1$, and $H_2$ is a copy of $G_2\}$. We also write $S_{l, k}(G)$ for short if $G_1$ is $K_l$ and $G_2$ is $K_k$.

Choose two vertex-disjoint complete graphs $K_l$ and $K_k$, and join the vertices in $K_l$ to the vertices of $K_k$ by edges arbitrarily. Denote the resulting graph by $H$. Let $G$ be another graph. For every $(H_1, H_2)\in S_{l, k}(G)$, let $f_G(H, (H_1, H_2))$ be the number of copies of $H$ in $G[V(H_1)\cup V(H_2)]$ such that $H_1$ corresponds to the $K_l$ part of $H$ and $H_2$ corresponds to the $K_k$ part of $H$. Then
$$
\mathcal{N}(H, G)=\frac{1}{a}\sum_{(H_1, H_2)\in S_{l, k}(G)}f_G(H, (H_1, H_2)),
$$
where $a$ is the number of times a copy of $H$ is counted, and depends only on $H$.

For example, given a $K_2$ and a $K_1$, if we join one vertex of $K_2$ and the vertex of $K_1$, then we get a path $P_3$. Denote this path by $H$. Let $G$ be a $K_3$, and $V(G)=\{a, b, c\}$. Then $f_G(H, (G[a, b], G[c]))=2$, and
$$
\mathcal{N}(H, G)=\frac{1}{2}(2+2+2)=3.
$$

\begin{proposition}\label{K1+Kl}
Choose two vertex-disjoint complete graphs $K_l$ and $K_1,$ and join the vertices in $K_l$ to the vertex of $K_1$ by edges arbitrarily. Denote the resulting graph by $H$, and assume that $H$ is not a copy of $K_{l+1}.$ Then $H$ is $k$-Tur\'an-good for $k\geq l+1.$
\end{proposition}
\begin{proof}
Let $G$ be a $K_k$-free graph with $n$ vertices having maximum number of copies of $H.$ We write $S=S_{l, 1}(G)$ and $T=S_{l, 1}(T_{k-1}(n))$ for simplicity. Let $s=|S|$ and $t=|T|.$ Then $s=\mathcal{N}(K_l, G)(n-l)$ and $t=\mathcal{N}(K_l, T_{k-1}(n))(n-l).$ Since $G$ is $K_k$-free, $\mathcal{N}(K_l, G)\leq\mathcal{N}(K_l, T_{k-1}(n)).$ And hence $s\leq t.$

Let $S=S_1\cup S_2$, where $S_1=\{(H_1, H_2):H_1$ and $H_2$ are disjoint subgraphs of $G, H_1$ is a copy of $K_l$, $H_2$ is a copy of $K_1$, and $G[V(H_1)\cup V(H_2)]$ is a copy of $K_{l+1}\}$ and $S_2=\{(H_1, H_2):H_1$ and $H_2$ are disjoint subgraphs of $G, H_1$ is a copy of $K_l$, $H_2$ is a copy of $K_1$, and $G[V(H_1)\cup V(H_2)]$ is not a copy of $K_{l+1}\}.$ Let $s_1=|S_1|$ and $s_2=|S_2|.$ Then $s=s_1+s_2.$ Define $T_1$, $T_2$, $t_1$, and $t_2$ similarly.

Note that $s_1=\mathcal{N}(K_{l+1}, G)(l+1)$, and $t_1=\mathcal{N}(K_{l+1}, T_{k-1}(n))(l+1).$ So $t_1\geq s_1.$ Thus there is an injection $\phi_1:S_1\rightarrow T_1$ such that if $\phi_1(H_1, H_2)=(H_1', H_2')$, then $G[V(H_1)\cup V(H_2)]$ is isomorphic to a subgraph of $T_{k-1}(n)[V(H_1')\cup V(H_2')]$ (indeed, both of them are $K_{l+1}$).

Now consider $S_2$ and $T_2.$ For every $(H_1, H_2)\in S_2$, since $H_1$ is a copy of $K_{l}$ but $G[V(H_1)\cup V(H_2)]$ is not a copy of $K_{l+1}$, it implies that in $G,$ the vertex in $H_2$ is not adjacent to at least one vertex in $H_1.$ On the other hand, in $T_{k-1}(n),$ for every $(H_1', H_2')\in T_2$, the vertex in $H_2'$ is not adjacent to exactly one vertex in $H_1'.$ Note that
$$
|S_2|=s_2=s-s_1\leq t-s_1=(t_1-s_1)+t_2=|(T_1\backslash\phi_1(S_1))\cup T_2|.
$$
So there is an injection $\phi_2:S_2\rightarrow (T_1\backslash\phi_1(S_1))\cup T_2$ such that if $\phi_2(H_1, H_2)=(H_1', H_2')$, then $G[V(H_1)\cup V(H_2)]$ is isomorphic to a subgraph of $T_{k-1}(n)[V(H_1')\cup V(H_2')].$

Combining $\phi_1$ and $\phi_2$, we find an injection $\phi$ from $S$ to $T$, such that if $\phi(H_1, H_2)=(H_1', H_2')$, then $G[V(H_1)\cup V(H_2)]$ is isomorphic to a subgraph of $T_{k-1}(n)[V(H_1')\cup V(H_2')].$ Therefore,
\begin{equation*}
  \begin{split}
  \mathcal{N}(H, G)&=\frac{1}{a}\sum_{(H_1, H_2)\in S}f_G(H, (H_1, H_2))\\
                   &\leq\frac{1}{a}\sum_{\phi(H_1, H_2)\in \phi(S)}f_{T_{k-1}(n)}(H, \phi(H_1, H_2))\\
                   &\leq\frac{1}{a}\sum_{(H_1', H_2')\in T}f_{T_{k-1}(n)}(H, (H_1', H_2'))\\
                   &=\mathcal{N}(H, T_{k-1}(n)).
   \end{split}
\end{equation*}
This completes the proof.
\end{proof}

By the above proposition, we can get the following corollary, which was proved in \cite{2020arXiv200603756G}.
\begin{corollary}
  $P_3$ is $k$-Tur\'an-good for $k\ge3.$
\end{corollary}

In order to prove that $P_5$ is $k$-Tur\'an-good later, we need the following proposition.
\begin{proposition}\label{Kl+Km}
  For any positive integers $l, m, k$ with $2\leq l\leq m<k$, let $H$ be the graph such that $V(H)=V(K_l)\cup V(K_m)$ and $E(H)=E(K_l)\cup E(K_m)$. Then $H$ is $k$-Tur\'an-good.
\end{proposition}
The proof is similar to that for proving $M_l$ is $k$-Tur\'an-good in \cite{2020arXiv200616150G}. For the sake of completeness, we include it here.
\begin{proof}
  Let $G$ be an $n$-vertex $K_k$-free graph with the largest number of copies of $H$, and let $n$ be sufficiently large with respect to $k$.

  $\mathbf{Case~ 1.}$ $G$ has chromatic number at least $k.$ We claim that
  \begin{equation}\label{xie1}
    \mathcal{N}(K_l, T_{k-1}(n))\mathcal{N}(K_m, T_{k-1}(n-l))-\mathcal{N}(H, G)=\Omega(n^{l+m-1}),
  \end{equation}
  and
  \begin{equation}\label{xie2}
    \mathcal{N}(K_l, T_{k-1}(n))\mathcal{N}(K_m, T_{k-1}(n-l))-\mathcal{N}(H, T_{k-1}(n))=O(n^{l+m-2}).
  \end{equation}
  This shows that $\mathcal{N}(H, G)<\mathcal{N}(H, T_{k-1}(n))$ for large $n$, a contradiction.

  To prove (\ref{xie1}), we need a theorem of Erd\H{o}s and Simonovits \cite{MR342429}, which states that for every $n$-vertex $K_k$-free graph with chromatic number at least $k$, there is a vertex $v$ of degree at most $(1-\frac{1}{k-7/3})n$. Then $\mathcal{N}(K_l, G)$ is equal to the sum of the number of $K_l$ in $G$ that contain $v$ as a vertex and the number of $K_l$ in $G$ that do not contain $v$ as a vertex, where the second term is at most $\mathcal{N}(K_l, T_{k-1}(n-1))$. As for the number of $K_l$ in $G$ that contain $v$ as a vertex, let $G_1$ be the subgraph of $G$ induced by the neighbors of $v$. Then $G_1$ is $K_{k-1}$-free, and by Theorem \ref{Kr is k turan good}, $\mathcal{N}(K_{l-1}, G_1)\leq{k-2\choose l-1}(\lceil\frac{(1-\frac{1}{k-7/3})n}{k-2}\rceil)^{l-1}$. So
  $$
  \mathcal{N}(K_l, G)\leq{k-2\choose l-1}\left(\lceil\frac{(1-\frac{1}{k-7/3})n}{k-2}\rceil\right)^{l-1}+\mathcal{N}(K_l, T_{k-1}(n-1)).
  $$
  On the other hand, we can delete one vertex from $T_{k-1}(n)$ to obtain $T_{k-1}(n-1)$, and
  $$
  \mathcal{N}(K_l, T_{k-1}(n))\geq\binom{k-2}{l-1}\left(\lfloor\frac{n}{k-1}\rfloor\right)^{l-1}+\mathcal{N}(K_l, T_{k-1}(n-1)).
  $$
  Since $\frac{(1-\frac{1}{k-7/3})}{k-2}<\frac{1}{k-1}$, it implies that $\mathcal{N}(K_l, T_{k-1}(n))-\mathcal{N}(K_l, G)\geq\Theta(n^{l-1})$, and we write $\mathcal{N}(K_l, T_{k-1}(n))-\mathcal{N}(K_l, G)=\Omega(n^{l-1})$.

  Let $G_2$ be any resulting graph after deleting a copy of $K_l$ from $G$. Then $G_2$ is an $(n-l)$-vertex $K_k$-free graph. So $\mathcal{N}(K_m, G_2)\leq\mathcal{N}(K_m, T_{k-1}(n-l))$. Therefore
  \begin{equation*}
  \begin{split}
    &\mathcal{N}(K_l, T_{k-1}(n))\mathcal{N}(K_m, T_{k-1}(n-l))-\mathcal{N}(H, G)\\
     \geq&\mathcal{N}(K_l, T_{k-1}(n))\mathcal{N}(K_m, T_{k-1}(n-l))-\mathcal{N}(K_l, G)\mathcal{N}(K_m, T_{k-1}(n-l)) \\
      =&\Omega(n^{l+m-1}).
  \end{split}
  \end{equation*}
  This proves (\ref{xie1}).

  Now let us prove (\ref{xie2}). We need to count the number of copies of $H$ in $T_{k-1}(n).$ We first pick a copy of $K_l$ in $T_{k-1}(n)$, and there are $\mathcal{N}(K_l, T_{k-1}(n))=\Theta(n^l)$ ways. After deleting the vertices of the picked $K_l$, the remaining graph $G_3$ is a complete $(k-1)$-partite graph on $(n-l)$ vertices. If $(k-1)\mid n$, then the size of every part is either $n/(k-1)$ or $n/(k-1)-1$; if $(k-1)\nmid n$, then the size of every part is one of $\{\lceil n/k\rceil, \lceil n/k\rceil-1, \lfloor n/k\rfloor-1\}.$ In the first case, $G_3$ is $T_{k-1}(n-l)$. In the second case, we can obtain $G_3$ from the $T_{k-1}(n-l)$ by moving a constant $c$ (indeed, $c\leq l$) number of vertices from some parts to other parts. In a word, we can always obtain $G_3$ from the $T_{k-1}(n-l)$ by moving a constant $c$ number of vertices from some parts to other parts. And it is easy to verify that each such move decreases the number of copies of $K_m$ by $O(n^{m-2}).$ Thus $\mathcal{N}(K_m,G_3)=\mathcal{N}(K_m,T_{k-1}(n))-O(n^{m-2}).$ Hence
  $$
  \mathcal{N}(H,T_{k-1}(n))=\mathcal{N}(K_l, T_{k-1}(n))(\mathcal{N}(K_m,T_{k-1}(n-l))-O(n^{m-2})).
  $$
  Therefore, we have that
  $$
  \mathcal{N}(K_l, T_{k-1}(n))\mathcal{N}(K_m, T_{k-1}(n-l))-\mathcal{N}(H, T_{k-1}(n))=O(n^{l+m-2}).
  $$
  This proves (\ref{xie2}).

  $\mathbf{Case~ 2.}$ If $\chi(G)\le k-1,$ then we can assume that $G$ is a complete $(k-1)$-partite graph, as adding edges will neither decrease the number of copies of $H$ nor make a $K_k$. Suppose $V(G)=V_1\cup V_2\cup\cdots\cup V_{k-1}$. We now show that making the graph more balanced does not decrease the number of copies of $H$. In other words, in order to maximize $\mathcal{N}(H, G)$, $||V_i|-|V_j||$ should be $0$ or $1$, for every $i, j$. If so, the Tur\'an graph $T_{k-1}(n)$ will have the maximum number of copies of $H$, and it will complete the proof.

  To derive a contradiction, and without loss of generality, assume that $V_1=\{u_1, u_2, \ldots, u_a\}$ has size $a$ and $V_2=\{v, w_1, w_2, \ldots, w_b\}$ has size $b+1$, such that $b-a\ge1.$ Let $G'$ be another complete $(k-1)$-partite graph with parts $V_1\cup\{v\}, V_2\setminus\{v\}, V_3, V_4, \cdots, V_{k-1}$. So $G'$ can be obtained from $G$ by moving $v$ from $V_2$ to $V_1$. It suffices to show that $\mathcal{N}(H, G)\leq\mathcal{N}(H, G')$.

  We say a copy of $H$ in $G$ is of $type~ \mathrm{I}$ if for some $i$, both $v$ and $u_i$ are in its $K_l$ part, and a copy of $H$ in $G$ is of $type~ \mathrm{II}$ if for some $i$, both $v$ and $u_i$ are in its $K_m$ part. Similarly, we say a copy of $H$ in $G'$ is of $type~ \mathrm{III}$ if for some $i$, both $v$ and $w_i$ are in its $K_l$ part, and a copy of $H$ in $G'$ is of $type~ \mathrm{IV}$ if for some $i$, both $v$ and $w_i$ are in its $K_m$ part. So by moving $v$ from $V_2$ part to $V_1$ part, we lose $H$ of $type~ \mathrm{I}$ and $type~ \mathrm{II}$, and at the same time, we gain $H$ of $type~ \mathrm{III}$ and $type~ \mathrm{IV}$. We claim that the number of $type~ \mathrm{I}$ $H$s is less than or equal to the number of $type~ \mathrm{III}$ $H$s. If the claim is true, then similarly, the number of $type~ \mathrm{II}$ $H$s is less than or equal to the number of $type~ \mathrm{IV}$ $H$s. And this means that $\mathcal{N}(H, G)\leq\mathcal{N}(H, G')$.

  Let $A_\lambda$ $(1\leq\lambda\leq a)$ be the number of copies of $H$ of $type~ \mathrm{I}$ such that $v$ and $u_\lambda$ are in the $K_l$ part. Let $M=\mathcal{N}(K_{l-2}, G[V_3\cup V_4\cup\cdots\cup V_{k-1}])$ be the number of copies of $K_{l-2}$ in $G[V_3\cup V_4\cup\cdots\cup V_{k-1}]$, and let $K^{(i)}_{l-2}(1\leq i\leq M)$ be the copies of $K_{l-2}$ in $G[V_3\cup V_4\cup\cdots\cup V_{k-1}]$. Then
  $$
  A_\lambda=\sum_{i=1}^M\mathcal{N}\left(K_m, G[(V_1\cup V_2\cup\cdots\cup V_{k-1})\setminus(\{v, u_\lambda\}\cup V(K^{(i)}_{l-2}))]\right),
  $$
  and the number of $type~ \mathrm{I}$ $H$s in $G$ is $\sum_{\lambda=1}^aA_\lambda$.

  Let $B_\mu$ $(1\le\mu\le b)$ be the number of copies of $H$ of $type~ \mathrm{III}$ such that $u$ and $w_\mu$ are in the $K_l$ part. Recall that $G[V_3\cup V_4\cup\cdots\cup V_{k-1}]=G'[V_3\cup V_4\cup\cdots\cup V_{k-1}]$. Then
  $$
  B_\mu=\sum_{i=1}^M\mathcal{N}\left(K_m, G'[(V_1\cup V_2\cup\cdots\cup V_{k-1})\setminus(\{v, w_\mu\}\cup V(K^{(i)}_{l-2}))]\right),
  $$
  and the number of $type~ \mathrm{III}$ $H$s in $G'$ is $\sum_{\mu=1}^bB_\mu$.

  Fix $\lambda, \mu$ and $i$. Consider $\mathcal{N}\left(K_m, G[(V_1\cup V_2\cup\cdots\cup V_{k-1})\setminus(\{v, u_\lambda\}\cup V(K^{(i)}_{l-2}))]\right)$ and \\$\mathcal{N}\left(K_m, G'[(V_1\cup V_2\cup\cdots\cup V_{k-1})\setminus(\{v, w_\mu\}\cup V(K^{(i)}_{l-2}))]\right)$. Note that $G'[(V_1\cup V_2\cup\cdots\cup V_{k-1})\setminus(\{v, w_\mu\}\cup V(K^{(i)}_{l-2}))]$ is more balanced than $G[(V_1\cup V_2\cup\cdots\cup V_{k-1})\setminus(\{v, u_\lambda\}\cup V(K^{(i)}_{l-2}))]$. So
  \begin{equation*}
    \begin{split}
        & \mathcal{N}\left(K_m, G[(V_1\cup V_2\cup\cdots\cup V_{k-1})\setminus(\{v, u_\lambda\}\cup V(K^{(i)}_{l-2}))]\right) \\
        \leq &\mathcal{N}\left(K_m, G'[(V_1\cup V_2\cup\cdots\cup V_{k-1})\setminus(\{v, w_\mu\}\cup V(K^{(i)}_{l-2}))]\right).
     \end{split}
  \end{equation*}
  To be precise, one can calculate that
  \begin{equation*}
    \begin{split}
       & \mathcal{N}\left(K_m, G'[(V_1\cup V_2\cup\cdots\cup V_{k-1})\setminus(\{v, w_\mu\}\cup V(K^{(i)}_{l-2}))]\right)\\
       =& \mathcal{N}\left(K_m, G'[(V_3\cup V_4\cup\cdots\cup V_{k-1})\setminus V(K^{(i)}_{l-2})]\right)\\
        &+(a+b-1)\mathcal{N}\left(K_{m-1}, G'[(V_3\cup V_4\cup\cdots\cup V_{k-1})\setminus V(K^{(i)}_{l-2})]\right)\\
       &+a(b-1)\mathcal{N}\left(K_{m-2}, G'[(V_3\cup V_4\cup\cdots\cup V_{k-1})\setminus V(K^{(i)}_{l-2})]\right),
    \end{split}
  \end{equation*}
  and
  \begin{equation*}
    \begin{split}
       & \mathcal{N}\left(K_m, G[(V_1\cup V_2\cup\cdots\cup V_{k-1})\setminus(\{v, w_\mu\}\cup V(K^{(i)}_{l-2}))]\right)\\
       =& \mathcal{N}\left(K_m, G[(V_3\cup V_4\cup\cdots\cup V_{k-1})\setminus V(K^{(i)}_{l-2})]\right)\\
        &+(a+b-1)\mathcal{N}\left(K_{m-1}, G[(V_3\cup V_4\cup\cdots\cup V_{k-1})\setminus V(K^{(i)}_{l-2})]\right)\\
       &+(a-1)b\mathcal{N}\left(K_{m-2}, G[(V_3\cup V_4\cup\cdots\cup V_{k-1})\setminus V(K^{(i)}_{l-2})]\right).
    \end{split}
  \end{equation*}
  Since $b\geq a+1$, it implies that
    \begin{equation*}
    \begin{split}
        & \mathcal{N}\left(K_m, G[(V_1\cup V_2\cup\cdots\cup V_{k-1})\setminus(\{v, u_\lambda\}\cup V(K^{(i)}_{l-2}))]\right) \\
        \leq &\mathcal{N}\left(K_m, G'[(V_1\cup V_2\cup\cdots\cup V_{k-1})\setminus(\{v, w_\mu\}\cup V(K^{(i)}_{l-2}))]\right).
     \end{split}
  \end{equation*}

  The inequality holds for every $\lambda, \mu$ and $i$. Hence
  \begin{equation*}
    \begin{split}
       \sum_{\lambda=1}^aA_\lambda & =\sum_{\lambda=1}^a\sum_{i=1}^M\mathcal{N}(K_m, G[V_1\cup V_2\cup\cdots\cup V_{k-1}\setminus(\{v, u_\lambda\}\cup V(K^{(i)}_{l-2}))]) \\
         & \leq\sum_{\mu=1}^b\sum_{i=1}^M\mathcal{N}(K_m, G'[V_1\cup V_2\cup\cdots\cup V_{k-1}\setminus(\{v, w_\mu\}\cup V(K^{(i)}_{l-2}))])\\
         &=\sum_{\mu=1}^bB_\mu.
    \end{split}
  \end{equation*}
  Therefore, the number of $type~ \mathrm{I}$ $H$s in $G$ is less than or equal to the number of $type~ \mathrm{III}$ $H$s in $G'$, as claimed.
\end{proof}
\begin{remark}
Combining this proposition with Proposition \ref{K1+Kl}, it implies that the conclusion holds for every $l, m$ and $k$ such that $1\leq l\leq m<k$.

This proposition can be extended to a finite number of cliques, i.e. for $1 \leq l_1 \leq l_2 \leq\cdots\leq l_m <k$, if a graph H satisfies that $V(H)=V(K_{l_1})\cup V(K_{l_2})\cup\cdots\cup V(K_{l_m})$ and $E(H)=E(K_{l_1})\cup E(K_{l_2})\cup\cdots\cup E(K_{l_m})$, then $H$ is $k$-Tur\'an-good.
\end{remark}

\begin{corollary}\label{K3+K2 is k turan good}
\FlagGraph{5}{1--2,3--4,4--5,3--5} is $k$-Tur\'an-good for $k\geq4$.
\end{corollary}
Now, we can prove the following proposition.
\begin{proposition}\label{K2+Kk-2}
Choose two vertex-disjoint complete graphs $K_{k-2}$ and $K_2$, and join the vertices in $K_{k-2}$ to the vertices of $K_2$ by edges arbitrarily. Denote the resulting graph by $H$, and assume that $H$ is not a copy of $K_k.$ Then $H$ is $k$-Tur\'an-good.
\end{proposition}
\begin{proof}
Let $G$ be a $K_k$-free graph with $n$ vertices having maximum number of $H$. We write $S=S_{k-2, 2}(G)$ and $T=S_{k-2, 2}(T_{k-1}(n))$ for simplicity. Let $s=|S|$ and $t=|T|.$ Proposition \ref{Kl+Km} shows that $s\leq t.$

Let $S=S_1\cup S_2\cup S_3$, where $S_1=\{(H_1, H_2):H_1$ (with vertex set $\{a_1, a_2, \ldots, a_{k-2}\}$) and $H_2$ (with vertex set $\{b_1, b_2\}$) are disjoint subgraphs in $G, G[V(H_1)\cup\{b_1\}]$ is a copy of $K_{k-1}$, and $H_2$ is a copy of $K_2\}$, $S_2=\{(H_1, H_2):H_1$ (with vertex set $\{a_1, a_2, \ldots, a_{k-2}\}$) and $H_2$ (with vertex set $\{b_1, b_2\}$) are disjoint subgraphs in $G, H_1$ is a copy of $K_{k-2}, H_2$ is a copy of $K_2, G[a_1, a_2, \ldots, a_{k-3}, b_1, b_2]$ is a copy of $K_{k-1}, G[a_1, a_2, \ldots, a_{k-2}, b_1]$ is not a copy of $K_{k-1}$, and $G[a_1, a_2, \ldots, a_{k-2}, b_2]$ is not a copy of $K_{k-1}\}$, and $S_3=\{(H_1, H_2):H_1$ and $H_2$ are disjoint subgraphs in $G, H_1$ is a copy of $K_{k-2}, H_2$ is a copy of $K_2$, and $G[V(H_1)\cup V(H_2)]$ contains no $K_{k-1}\}.$ Let $s_1=|S_1|$, $s_2=|S_2|$, and $s_3=|S_3|.$ Then $s=s_1+s_2+s_3.$

Let $T=T_1\cup T_2$, where $T_1=\{(H_1, H_2):H_1$ (with vertex set $\{a_1, a_2, \ldots, a_{k-2}\}$) and $H_2$ (with vertex set $\{b_1, b_2\}$) are disjoint subgraphs in $T_{k-1}(n), T_{k-1}(n)[V(H_1)\cup\{b_1\}]$ is a copy of $K_{k-1}$, and $H_2$ is a copy of $K_2\}$, and $T_2=\{(H_1, H_2):H_1$ and $H_2$ are disjoint subgraphs in $T_{k-1}(n), H_1$ is a $K_{k-2}, H_2$ is a $K_2$, and $T_{k-1}(n)[V(H_1)\cup V(H_2)]$ contains no $K_{k-1}\}.$ Let $t_1=|T_1|$ and $t_2=|T_2|.$ Then $t=t_1+t_2.$

Let $K^{(i)}(1\leq i\leq m)$ be the different copies of $K_{k-1}$ in $G.$ For every $i$, define $\alpha_j^{(i)}=|\{v\in V(G)\backslash V(K^{(i)}): v$ is adjacent to exactly $j$ vertices of $K^{(i)}\}|.$ Then $\sum_{j=1}^{k-2}\alpha_j^{(i)}\leq n-(k-1)$ holds for every $i.$ It is also easy to verify that
$$
s_1=\sum_{i=1}^m\sum_{j=1}^{k-2}j\alpha_j^{(i)},
$$
and
$$
s_2=\sum_{i=1}^m\alpha_{k-3}^{(i)}.
$$
Thus
$$
s_1+s_2\leq(k-2)\sum_{i=1}^m\sum_{j-1}^{k-2}\alpha_{j}^{(i)}\leq(k-2)m(n-(k-1)).
$$
Note that $m$ is the total number of copies of $K_{k-1}$ in $G$, i.e. $m=\mathcal{N}(K_{k-1}, G).$ So
$$
s_1+s_2\leq(k-2)\mathcal{N}(K_{k-1}, G)(n-(k-1))\leq(k-2)\mathcal{N}(K_{k-1}, T_{k-1}(n))(n-(k-1))=t_1.
$$
Thus there is an injection $\phi_1:S_1\cup S_2\rightarrow T_1$ such that if $\phi_1(H_1, H_2)=(H_1', H_2')$, then $G[V(H_1)\cup V(H_2)]$ is isomorphic to a subgraph of $T_{k-1}(n)[V(H_1')\cup V(H_2')]$ (indeed, the graph $T_{k-1}(n)[V(H_1')\cup V(H_2')]$ has $\binom{k}{2}-1$ edges, which is the largest $k$-vertex $K_k$-free graph).

Now consider $S_3$ and $T_2.$ For every $(H_1, H_2)\in S_3,$ let $V(H_1)=\{a_1, a_2, \ldots, a_{k-2}\}$ and $V(H_2)=\{b_1, b_2\}$. Since $H_1$ is a copy of $K_{k-2},$ $H_2$ is a copy of $K_2,$ and $G[a_1, a_2, \ldots, a_{k-2}, b_1, b_2]$ contains no $K_{k-1},$ it implies that there are distinct $a_{i_1}$ and $a_{i_2},$ such that $b_1$ is not adjacent to $a_{i_1}$ and $b_2$ is not adjacent to $a_{i_2},$ in $G.$ On the other hand, for every $(H_1', H_2')\in T_2,$ let $V(H_1')=\{a_1', a_2', \ldots, a_{k-2}'\}$ and $V(H_2')=\{b_1', b_2'\},$ $b_1'$ is not adjacent to exactly one of $a_i'$ $(1\leq i\leq l),$ and $b_2'$ is not adjacent to exactly another one of $a_i'$ $(1\leq i\leq l),$ in $T_{k-1}(n).$ Note that
$$
|S_3|=s_3=s-s_1-s_2\leq t-s_1-s_2=(t_1-s_1-s_2)+t_2=|(T_1\backslash\phi_1(S_1\cup S_2))\cup T_2|.
$$
So there is an injection $\phi_2:S_3\rightarrow (T_1\backslash\phi_1(S_1\cup S_2))\cup T_2$ such that if $\phi_2(H_1, H_2)=(H_1', H_2'),$ then $G[V(H_1)\cup V(H_2)]$ is isomorphic to a subgraph of $T_{k-1}(n)[V(H_1')\cup V(H_2')].$

Combining $\phi_1$ and $\phi_2,$ we find an injection $\phi$ from $S$ to $T,$ such that if $\phi(H_1, H_2)=(H_1', H_2'),$ then $G[V(H_1)\cup V(H_2)]$ is isomorphic to a subgraph of $T_{k-1}(n)[V(H_1')\cup V(H_2')].$ Therefore,
\begin{equation*}
  \begin{split}
     \mathcal{N}(H, G)&=\frac{1}{a}\sum_{(H_1, H_2)\in S}f_G(H, (H_1, H_2))\\
                      &\leq\frac{1}{a}\sum_{\phi(H_1, H_2)\in \phi(S)}f_{T_{k-1}(n)}(H, \phi(H_1, H_2))\\
                      &\leq\frac{1}{a}\sum_{(H_1', H_2')\in T}f_{T_{k-1}(n)}(H, (H_1', H_2'))\\
                      &=\mathcal{N}(H, T_{k-1}(n)).
   \end{split}
\end{equation*}
This completes the proof.
\end{proof}

By the above proposition, we can get the following corollary which has been proved in \cite{2020arXiv200616150G}.
\begin{corollary}
  $P_4$ is $4$-Tur\'an-good.
\end{corollary}
Also, we have the following corollary which will be used later.
\begin{corollary}\label{two small 5 turan good graphs}
  \FlagGraph{5}{1--2,2--3,1--3,1--4,4--5,1--5} is $5$-Tur\'an-good.
\end{corollary}

\section{$P_4$ is $k$-Tur\'an-good}

In this section, we prove the following theorem.
\begin{theorem}\label{thm of P4 k turan good}
  $P_4$ is $k$-Tur\'an-good for $k\ge5.$
\end{theorem}
\begin{proof}
Given an integer $k\ge5,$ let $G$ be an $n$-vertex graph which is $K_k$-free. The graphs in the first row of Table \ref{4 vertex induced graph table} are all possible induced $4$-vertex subgraphs which contain at least one $P_4.$ We assume that the number of them in $G$ are as that in the second row. And by an easy counting, we get all the values of the table.

\begin{table}[h]
\caption{Induced $4$-vertex subgraphs containing $P_4$}\label{4 vertex induced graph table}
\begin{center}
\begin{tabular}{|c|c|c|c|c|c|c| p{5cm}|}
\hline
Induced $4$-vertex subgraphs containing $P_4$&\FlagGraph{4}{1--2,2--3,3--4} &\FlagGraph{4}{1--2,2--3,3--4,4--1} &
 \FlagGraph{4}{1--2,2--4,4--1,2--3} & \FlagGraph{4}{1--2,2--3,3--1,4--1,4--2}& \FlagGraph{4}{1--2,1--3,1--4,2--3,2--4,3--4} \\ \hline
 Number of each subgraph in $G$& $x_1$ & $x_2$&$x_3$ &$x_4$&$x_5$ \\ \hline
 Number of $M_2$ in each subgraph &1&2&1&2&3 \\ \hline
 Number of $\FlagGraph{4}{1--2,2--3,1--3}$ in each subgraph &0&0&1&2&4 \\ \hline
 Number of $K_4$ in each subgraph &0&0&0&0&1 \\ \hline
 Number of $P_4$ in each subgraph &1&4&3&6&12 \\
\hline
\end{tabular}
\end{center}
\end{table}

Since $K_3$ is $k$-Tur\'an-good by Theorem \ref{Kr is k turan good}, it follows directly that \FlagGraph{4}{1--2,2--3,1--3} is also $k$-Tur\'an-good.
Because of $M_2,$ \FlagGraph{4}{1--2,2--3,1--3} and $K_4$ are $k$-Tur\'an-good, we have
\begin{equation*}
\left\{
\begin{array}{cll}
  x_1+2x_2+x_3+2x_4+3x_5 &\le\mathcal{N}(M_2,G) &\le  \mathcal{N}(M_2,T_{k-1}(n)), \\
  x_3+2x_4+4x_5 &\le\mathcal{N}(\FlagGraph{4}{1--2,2--3,1--3},G) &\le  \mathcal{N}(\FlagGraph{4}{1--2,2--3,1--3},T_{k-1}(n)), \\
  x_5 &\le \mathcal{N}(K_4,G) &\le \mathcal{N}(K_4,T_{k-1}(n)).
\end{array}
\right.
\end{equation*}

Therefore,
\begin{equation*}
\begin{aligned}
\mathcal{N}(P_4,G) &=x_1+4x_2+3x_3+6x_4+12x_5\\
                   &\le2(x_1+2x_2+x_3+2x_4+3x_5)+(x_3+2x_4+4x_5)+2x_5\\
                   &\le2\mathcal{N}(M_2,T_{k-1}(n))+\mathcal{N}(\FlagGraph{4}{1--2,2--3,1--3},T_{k-1}(n))+2\mathcal{N}(K_4,T_{k-1}(n)).
\end{aligned}
\end{equation*}
Now we only need to show that $2\mathcal{N}(M_2,T_{k-1}(n))+\mathcal{N}(\FlagGraph{4}{1--2,2--3,1--3},T_{k-1}(n))+2\mathcal{N}(K_4,T_{k-1}(n))=\mathcal{N}(P_4,T_{k-1}(n)).$ Indeed, by easy counting, we have the following
\begin{equation*}
  \left\{
  \begin{array}{cl}
    \mathcal{N}(P_4, T_{k-1}(n)) &= 12\mathcal{N}(K_4, T_{k-1}(n))+6\mathcal{N}(K_{1,1,2}, T_{k-1}(n))+4\mathcal{N}(K_{2,2}, T_{k-1}(n)),  \\
    \mathcal{N}(\FlagGraph{4}{1--2,2--3,1--3}, T_{k-1}(n))&= 4\mathcal{N}(K_4, T_{k-1}(n)) + 2\mathcal{N}(K_{2,2}, T_{k-1}(n)),\\
    \mathcal{N}(M_2, T_{k-1}(n)) &=3\mathcal{N}(K_4, T_{k-1}(n)) +2\mathcal{N}(K_{1,1,2}, T_{k-1}(n))+2\mathcal{N}(K_{2,2}, T_{k-1}(n)).
  \end{array}
  \right.
\end{equation*}
With the above three equalities, it can be derived that
$$2\mathcal{N}(M_2,T_{k-1}(n))+\mathcal{N}(\FlagGraph{4}{1--2,2--3,1--3},T_{k-1}(n))+2\mathcal{N}(K_4,T_{k-1}(n))=\mathcal{N}(P_4,T_{k-1}(n)).$$
This completes the proof.

\end{proof}

\section{$P_5$ is $k$-Tur\'an-good}

In this section, we will prove that $P_5$ is $k$-Tur\'an-good for all $k\ge4.$ The case $k=4$ is much easier, so we prove it first.
\begin{theorem}
  $P_5$ is $4$-Tur\'an-good.
\end{theorem}
\begin{proof}
  Let $G$ be an $n$-vertex graph which is $K_4$-free. Similar to the proof of Theorem \ref{thm of P4 k turan good}, in Table \ref{5 vertex induced graph table}, we list all possible induced $5$-vertex subgraphs containing at least one $P_5,$ and assume that each has $x_i$ copies in $G$ for $1\le i\le 13.$

 \begin{table}[h]
\caption{Induced $5$-vertex subgraphs of $K_4$-free graphs containing $P_5$}\label{5 vertex induced graph table}
\begin{center}
\resizebox{\textwidth}{!}{
\begin{tabular}{|c|c|c|c|c|c|c|c|c|c|c|c|c|c| p{5cm}|}
\hline
\makecell{Induced $5$-vertex \\subgraphs containing $P_5$}& \FlagGraph{5}{1--2,2--3,3--4,4--5}& \FlagGraph{5}{1--2,2--3,3--4,4--5,1--5}&\FlagGraph{5}{1--2,2--3,3--4,4--5,5--3}&
\FlagGraph{5}{1--2,2--3,3--4,4--1,1--5}&\FlagGraph{5}{1--2,2--3,3--4,4--1,1--5,1--3}&
\FlagGraph{5}{1--2,2--3,3--4,4--1,1--5,2--4}&\FlagGraph{5}{1--2,2--3,3--4,2--5,3--5}&
\FlagGraph{5}{1--2,2--3,3--4,4--5,5--1,2--5}&\FlagGraph{5}{1--2,2--3,3--4,4--5,5--1,2--5,2--4}&
\FlagGraph{5}{1--2,2--3,3--4,4--5,1--5,1--3,2--5}& \FlagGraph{5}{1--3,1--4,1--5,2--3,2--4,2--5}&
\FlagGraph{5}{1--2,1--3,1--4,1--5,2--3,2--4,2--5}&\FlagGraph{5}{1--2,1--3,1--4,1--5,2--4,2--5,3--4,3--5}\\ \hline
 Number of each subgraph in $G$& $x_1$ & $x_2$&$x_3$ &$x_4$&$x_5$&$x_6$&$x_7$&$x_8$&$x_9$&$x_{10}$&$x_{11}$&$x_{12}$&$x_{13}$ \\ \hline
 Number of $\FlagGraph{5}{1--2,3--4}$ in each subgraph &3&5&4&4&4&5&3&6&7&8&6&6&10 \\ \hline
 Number of \FlagGraph{5}{1--2,3--4,4--5,5--3} in each subgraph &0&0&1&0&0&1&0&1&2&2&0&0&4 \\ \hline
 Number of \FlagGraph{5}{1--2,2--3,1--3,1--4,5--1,4--5} in each subgraph &0&0&0&0&0&0&0&0&1&0&0&0&2 \\ \hline
 Number of $P_5$ in each subgraph & 1&5&2&2&2&4&1&7&9&12&6&6&24\\
\hline
\end{tabular}}
\end{center}
\end{table}

  Now we have that
  \begin{equation*}
    \mathcal{N}(P_5,G)=1x_1+5x_2+2x_3+2x_4+2x_5+4x_6+1x_7+7x_8+9x_{9}+12x_{10}+6x_{11}+6x_{12}+24x_{13}.
  \end{equation*}

  By Theorem \ref{thm of H+K-1_is k turan good}, we know that \FlagGraph{5}{1--2,3--4,4--5,5--3} and \FlagGraph{5}{1--2,2--3,1--3,4--5,5--1,1--4} are $4$-Tur\'an-good. And since $M_2$ is $4$-Tur\'an-good, so is   \FlagGraph{5}{1--2,3--4}. Therefore
  \begin{equation*}
    \left\{
    \begin{array}{ccc}
      3x_1+5x_2+4x_3+4x_4+4x_5+5x_6+ \\
      3x_7+6x_8+7x_9+8x_{10}+6x_{11}+6x_{12}+10x_{13}&\le \mathcal{N}(\FlagGraph{5}{1--2,3--4},G) &\le \mathcal{N}(\FlagGraph{5}{1--2,3--4},T_3(n)),  \\
       1x_3+1x_6+1x_8+2x_9+2x_{10}+4x_{13}&\le \mathcal{N}(\FlagGraph{5}{1--2,3--4,4--5,5--3},G) &\le \mathcal{N}(\FlagGraph{5}{1--2,3--4,4--5,5--3},T_3(n)),\\
       1x_9+2x_{13}&\le \mathcal{N}(\FlagGraph{5}{1--2,2--3,1--3,4--5,5--1,1--4},G) &\le\mathcal{N}(\FlagGraph{5}{1--2,2--3,1--3,4--5,5--1,1--4},T_3(n)).
    \end{array}
    \right.
  \end{equation*}

  So we have that
\begin{equation*}
\begin{aligned}
  \mathcal{N}(P_5,G) &\le \mathcal{N}(\FlagGraph{5}{1--2,3--4},G)+3\mathcal{N}(\FlagGraph{5}{1--2,3--4,4--5,5--3},G)
  +\mathcal{N}(\FlagGraph{5}{1--2,2--3,1--3,4--5,5--1,1--4},G)\\
  &\le \mathcal{N}(\FlagGraph{5}{1--2,3--4},T_3(n))+3\mathcal{N}(\FlagGraph{5}{1--2,3--4,4--5,5--3},T_3(n))+
  \mathcal{N}(\FlagGraph{5}{1--2,2--3,1--3,4--5,5--1,1--4},T_3(n)).
\end{aligned}
\end{equation*}

Since
\begin{equation*}
  \left\{
  \begin{array}{cl}
    \mathcal{N}(P_5, T_3(n)) &= 6\mathcal{N}(K_{2,3}, T_3(n))+6\mathcal{N}(K_{1,1,3}, T_3(n))+24\mathcal{N}(K_{1,2,2}, T_3(n)),  \\
    \mathcal{N}(\FlagGraph{5}{1--2,3--4}, T_3(n))&= 6\mathcal{N}(K_{2,3}, T_3(n)) + 6\mathcal{N}(K_{1,1,3}, T_3(n))+10\mathcal{N}(K_{1,2,2},T_3(n)),\\
    \mathcal{N}(\FlagGraph{5}{1--2,3--4,4--5,5--3},T_3(n)) &=4\mathcal{N}(K_{1,2,2}, T_3(n)),\\
    \mathcal{N}(\FlagGraph{5}{1--2,2--3,1--3,4--5,5--1,1--4},T_3(n)) &=2\mathcal{N}(K_{1,2,2},T_3(n)),
  \end{array}
  \right.
\end{equation*}
it is easy to derive that
\begin{equation*}
  \mathcal{N}(P_5,T_3(n))=\mathcal{N}(\FlagGraph{5}{1--2,3--4},T_3(n))+3\mathcal{N}(\FlagGraph{5}{1--2,3--4,4--5,5--3},T_3(n))+
  \mathcal{N}(\FlagGraph{5}{1--2,2--3,1--3,4--5,5--1,1--4},T_3(n)).
\end{equation*}
Thus, we have that $\mathcal{N}(P_5,G)\le\mathcal{N}(P_5,T_3(n)),$ which means that $P_5$ is $4$-Tur\'an-good.

\end{proof}

To prove that $P_5$ is $k$-Tur\'an-good, we need the following proposition.

\begin{proposition}
  \FlagGraph{5}{1--2,1--3,2--3,1--4,1--5,4--5} is $k$-Tur\'an-good for $k\ge4.$
\end{proposition}
\begin{proof}
   The cases that  \FlagGraph{5}{1--2,1--3,2--3,1--4,1--5,4--5} is $4$-Tur\'an-good and $5$-Tur\'an-good follow from Theorem \ref{thm of H+K-1_is k turan good} and Corollary \ref{two small 5 turan good graphs}, respectively.

   Now we assume that $k\ge6.$ Let $G$ be an $n$-vertex $K_k$-free graph with maximal number of copies of \FlagGraph{5}{1--2,2--3,1--3,1--4,1--5,4--5}. As before, we list all the parameters in Table \ref{F(2) induced graph table}.

     \begin{table}[h]
\caption{Induced $5$-vertex subgraphs containing \FlagGraph{5}{1--2,1--3,2--3,1--4,1--5,4--5}}\label{F(2) induced graph table}
\begin{center}
\begin{tabular}{|c|c|c|c|c|c|c|c| p{5cm}|}
\hline
\makecell*[c]{Induced $5$-vertex \\subgraphs containing \FlagGraph{5}{1--2,1--3,2--3,1--4,1--5,4--5}}&\FlagGraph{5}{1--2,2--3,1--3,1--4,1--5,4--5} &\FlagGraph{5}{1--2,1--3,1--4,1--5,2--3,3--4,4--5} &
 \FlagGraph{5}{1--2,1--3,1--4,1--5,2--3,2--4,2--5} & \FlagGraph{5}{1--2,1--3,1--4,1--5,2--4,2--5,3--4,3--5}& \FlagGraph{5}{1--2,1--3,2--3,2--4,2--5,3--4,3--5,4--5} &\FlagGraph{5}{1--2,1--3,1--4,2--3,2--4,2--5,3--4,3--5,4--5} &\FlagGraph{5}{1--2,1--3,1--4,1--5,2--3,2--4,2--5,3--4,3--5,4--5} \\ \hline
 Number of each subgraph in $G$& $x_1$ & $x_2$&$x_3$ &$x_4$&$x_5$ &$x_6$&$x_7$ \\ \hline
 Number of \FlagGraph{5}{1--2,3--4,4--5,3--5} in each subgraph &2&2&0&4&3&6&10 \\ \hline
 Number of $\FlagGraph{5}{1--2,2--3,1--3,1--4,2--4,3--4}$ in each subgraph &0&0&0&0&1&2&5 \\ \hline
 Number of $K_5$ in each subgraph &0&0&0&0&0&0&1 \\ \hline
 Number of \FlagGraph{5}{1--2,1--3,2--3,1--4,1--5,4--5} in each subgraph &1&1&0&2&2&6&15\\
\hline
\end{tabular}
\end{center}
\end{table}

    Fix $k\ge6.$ Corollary \ref{K3+K2 is k turan good} says that \FlagGraph{5}{1--2,3--4,4--5,3--5} is $k$-Tur\'an-good. \FlagGraph{5}{1--2,1--3,1--4,2--3,2--4,3--4} is $k$-Tur\'an-good by Proposition \ref{K1+Kl} and $K_5$ is $k$-Tur\'an-good by Theorem \ref{Kr is k turan good}. Therefore,
   \begin{equation*}
   \begin{aligned}
     \mathcal{N}(\FlagGraph{5}{1--2,1--3,2--3,1--4,1--5,4--5},G) &\le x_1+x_2+2x_4+2x_5+6x_6+15x_7 \\
      &\le\frac{\mathcal{N}(\FlagGraph{5}{1--2,3--4,4--5,3--5},G)+3\mathcal{N}(\FlagGraph{5}{1--2,1--3,1--4,2--3,2--4,3--4},G)+5\mathcal{N}(K_5,G)}{2} \\
      &\le\frac{\mathcal{N}(\FlagGraph{5}{1--2,3--4,4--5,3--5},T_{k-1}(n))+3\mathcal{N}(\FlagGraph{5}{1--2,1--3,1--4,2--3,2--4,3--4},T_{k-1}(n))+5\mathcal{N}(K_5,T_{k-1}(n))}{2}.
   \end{aligned}
   \end{equation*}

  Since
  \begin{equation*}
  \left\{
  \begin{array}{cl}
   \mathcal{N}(\FlagGraph{5}{1--2,1--3,2--3,1--4,1--5,4--5},T_{k-1}(n))&=2\mathcal{N}(K_{1,2,2},T_{k-1}(n))+6\mathcal{N}(K_{1,1,1,2},T_{k-1}(n))+15\mathcal{N}(K_5,T_{k-1}(n)),\\
   \mathcal{N}(\FlagGraph{5}{1--2,3--4,4--5,3--5},T_{k-1}(n))  &=4\mathcal{N}(K_{1,2,2},T_{k-1}(n))+6\mathcal{N}(K_{1,1,1,2},T_{k-1}(n))+10\mathcal{N}(K_5,T_{k-1}(n)),\\
   \mathcal{N}(\FlagGraph{5}{1--2,1--3,1--4,2--3,2--4,3--4},T_{k-1}(n))&=2\mathcal{N}(K_{1,1,1,2},T_{k-1}(n))+5\mathcal{N}(K_5,T_{k-1}(n)),
   \end{array}
  \right.
\end{equation*}
it is easy to check that
\begin{equation*}
  \frac{\mathcal{N}(\FlagGraph{5}{1--2,3--4,4--5,3--5},T_{k-1}(n))+3\mathcal{N}(\FlagGraph{5}{1--2,1--3,1--4,2--3,2--4,3--4},T_{k-1}(n))+5\mathcal{N}(K_5,T_{k-1}(n))}{2}=\mathcal{N}(\FlagGraph{5}{1--2,1--3,2--3,1--4,1--5,4--5},T_{k-1}(n)).
\end{equation*}
This completes the proof.
\end{proof}

Now we can prove the following theorem.

\begin{theorem}
  $P_5$ is $k$-Tur\'an-good for any $k\ge 5.$
\end{theorem}

\begin{proof}
We first prove that $P_5$ is $k$-Tur\'an-good for $k\ge6.$ Let $G$ be an $n$-vertex $K_k$-free graph containing the most number of copies of $P_5$. We obtain Table \ref{5 vertex induced graph table3} by counting.

 \begin{table}[h]
\caption{Induced $5$-vertex subgraphs containing $P_5$}\label{5 vertex induced graph table3}
\begin{center}
\resizebox{\textwidth}{!}{
\begin{tabular}{|c|c|c|c|c|c|c|c|c|c|c|c|c|c|c|c|c|c| p{5cm}|}
\hline
\makecell*[c]{Induced $5$-vertex \\subgraphs containing $P_5$}& \FlagGraph{5}{1--2,2--3,3--4,4--5}& \FlagGraph{5}{1--2,2--3,3--4,4--5,1--5}&\FlagGraph{5}{1--2,2--3,3--4,4--5,5--3}&
\FlagGraph{5}{1--2,2--3,3--4,4--1,1--5}&\FlagGraph{5}{1--2,2--3,3--4,4--1,1--5,1--3}&
\FlagGraph{5}{1--2,2--3,3--4,4--1,1--5,2--4}&\FlagGraph{5}{1--2,2--3,3--4,2--5,3--5}&
\FlagGraph{5}{1--2,2--3,3--4,4--5,5--1,2--5}&\FlagGraph{5}{1--2,2--3,3--4,4--5,5--1,2--5,2--4}&
\FlagGraph{5}{1--2,2--3,3--4,4--5,1--5,1--3,2--5}& \FlagGraph{5}{1--3,1--4,1--5,2--3,2--4,2--5}&
\FlagGraph{5}{1--2,1--3,1--4,1--5,2--3,2--4,2--5}&\FlagGraph{5}{1--2,1--3,1--4,1--5,2--4,2--5,3--4,3--5}&
\FlagGraph{5}{1--2,2--3,2--4,2--5,3--4,3--5,4--5}&\FlagGraph{5}{1--2,2--3,2--4,2--5,3--4,3--5,4--5,1--5}&
\FlagGraph{5}{1--3,1--4,1--5,2--3,2--4,2--5,3--4,3--5,4--5}&\FlagGraph{5}{1--2,1--3,1--4,1--5,2--3,2--4,2--5,3--4,3--5,4--5}\\ \hline
 Number of each subgraph in $G$& $x_1$ & $x_2$&$x_3$ &$x_4$&$x_5$&$x_6$&$x_7$&$x_8$&$x_9$&$x_{10}$&$x_{11}$&$x_{12}$&$x_{13}$&$x_{14}$&$x_{15}$&$x_{16}$&$x_{17}$ \\ \hline
 \makecell*[c]{Number of $\FlagGraph{5}{1--2,3--4}$ \\in each subgraph} &3&5&4&4&4&5&3&6&7&8&6&6&10&6&9&12&15 \\ \hline
 \makecell*[c]{Number of \FlagGraph{5}{1--2,3--4,4--5,5--3}\\ in each subgraph} &0&0&1&0&0&1&0&1&2&2&0&0&4&1&3&6&10 \\ \hline
 \makecell*[c]{Number of \FlagGraph{5}{1--2,2--3,1--3,1--4,5--1,4--5} \\in each subgraph} &0&0&0&0&0&0&0&0&1&0&0&0&2&0&2&6&15 \\ \hline
 \makecell*[c]{Number of $P_5$\\ in each subgraph} & 1&5&2&2&2&4&1&7&9&12&6&6&24&6&18&36&60\\
\hline
\end{tabular}}
\end{center}
\end{table}

Therefore,
\begin{equation*}
\begin{aligned}
  \mathcal{N}(P_5,G)=&1x_1+5x_2+2x_3+2x_4+2x_5+4x_6+1x_7+7x_8+9x_{9} \\
                      &+12x_{10}+6x_{11}+6x_{12}+24x_{13}+6x_{14}+18x_{15}+36x_{16}+60x_{17}.
\end{aligned}
\end{equation*}

Since \FlagGraph{5}{1--2,3--4}, \FlagGraph{5}{1--2,3--4,4--5,5--3} and \FlagGraph{5}{1--2,2--3,1--3,4--5,5--1,1--4} are $k$-Tur\'an-good, we get the following inequalities.
  \begin{equation*}
    \left\{
    \begin{array}{ccc}
      3x_1+5x_2+4x_3+4x_4+4x_5+5x_6+3x_7+6x_8+7x_9\\      +8x_{10}+6x_{11}+6x_{12}+10x_{13}+6x_{14}+9x_{15}+12x_{16}+15x_{17}&\le \mathcal{N}(\FlagGraph{5}{1--2,3--4},G) &\le \mathcal{N}(\FlagGraph{5}{1--2,3--4},T_{k-1}(n)),  \\
       1x_3+1x_6+1x_8+2x_9+2x_{10}+4x_{13}+1x_{14}+3x_{15}+6x_{16}+10x_{17}&\le \mathcal{N}(\FlagGraph{5}{1--2,3--4,4--5,5--3},G) &\le \mathcal{N}(\FlagGraph{5}{1--2,3--4,4--5,5--3},T_{k-1}(n)),\\
       1x_9+2x_{13}+2x_{15}+6x_{16}+15x_{17}&\le \mathcal{N}(\FlagGraph{5}{1--2,2--3,1--3,4--5,5--1,1--4},G) &\le\mathcal{N}(\FlagGraph{5}{1--2,2--3,1--3,4--5,5--1,1--4},T_{k-1}(n)).
    \end{array}
    \right.
  \end{equation*}

Thus,
\begin{equation*}
  \begin{aligned}
    \mathcal{N}(P_5,G)&\le \mathcal{N}(\FlagGraph{5}{1--2,3--4},G)+3\mathcal{N}(\FlagGraph{5}{1--2,3--4,4--5,5--3},G)+\mathcal{N}(\FlagGraph{5}{1--2,2--3,1--3,4--5,5--1,1--4},G) \\
    &\le\mathcal{N}(\FlagGraph{5}{1--2,3--4},T_{k-1}(n))+3\mathcal{N}(\FlagGraph{5}{1--2,3--4,4--5,5--3},T_{k-1}(n))+
    \mathcal{N}(\FlagGraph{5}{1--2,2--3,1--3,4--5,5--1,1--4},T_{k-1}(n)).
  \end{aligned}
  \end{equation*}

  Since
\begin{equation*}
  \left\{
  \begin{array}{cl}
    \mathcal{N}(P_5, T_{k-1}(n)) =& 60\mathcal{N}(K_5,T_{k-1}(n))+36\mathcal{N}(K_{1,1,1,2},T_{k-1}(n)) \\
                                  &+24\mathcal{N}(K_{1,2,2},T_{k-1}(n))+6\mathcal{N}(K_{1,1,3},T_{k-1}(n))+6\mathcal{N}(K_{2,3},T_{k-1}(n)),  \\
    \mathcal{N}(\FlagGraph{5}{1--2,3--4}, T_{k-1}(n))=&15\mathcal{N}(K_5,T_{k-1}(n))+12\mathcal{N}(K_{1,1,1,2}, T_{k-1}(n)) \\
    &+ 10\mathcal{N}(K_{1,2,2}, T_{k-1}(n))+6\mathcal{N}(K_{1,1,3},T_{k-1}(n))+6\mathcal{N}(K_{2,3},T_{k-1}(n)),\\
    \mathcal{N}(\FlagGraph{5}{1--2,3--4,4--5,3--5}, T_{k-1}(n)) =&10\mathcal{N}(K_5,T_{k-1}(n))+6\mathcal{N}(K_{1,1,1,2}, T_{k-1}(n))+4\mathcal{N}({K_{1,2,2},T_{k-1}(n)}),\\
    \mathcal{N}(\FlagGraph{5}{1--2,2--3,1--3,4--5,5--1,1--4},T_{k-1}(n))=&15\mathcal{N}(K_5,T_{k-1}(n))
    +6\mathcal{N}(K_{1,1,1,2},T_{k-1}(n))+2\mathcal{N}(K_{1,2,2},T_{k-1}(n)),
  \end{array}
  \right.
\end{equation*}
it is easy to check that
\begin{equation*}
  \mathcal{N}(P_5,T_{k-1}(n))=\mathcal{N}(\FlagGraph{5}{1--2,3--4},T_{k-1}(n))+3\mathcal{N}(\FlagGraph{5}{1--2,3--4,4--5,5--3},T_{k-1}(n))+
    \mathcal{N}(\FlagGraph{5}{1--2,2--3,1--3,4--5,5--1,1--4},T_{k-1}(n)),
\end{equation*}
which completes the proof.

 When $k=5,$ nothing changes except $x_{17}=0$ in Table \ref{5 vertex induced graph table3} and $\mathcal{N}(K_5,T_4(n))=0,$ the result that $P_5$ is $5$-Tur\'an-good follows by the same analysis.

\end{proof}

\section{Concluding remarks}
In this paper, we construct some new kinds of $k$-Tur\'an-good graphs. Based on the fact that some certain small graphs are $k$-Tur\'an-good, we can prove that $P_4$ and $P_5$ are $k$-Tur\'an-good. This may also work for $P_l$ when $l\ge6,$ but the number of $l$-vertex graphs containing at least one $P_l$ is too large to handle. Some new ideas to attack the problem are expected.

\bibliographystyle{abbrv}
\bibliography{REF}

\end{document}